\documentclass[12pt,a4paper,final]{amsart}
\usepackage[cp1250]{inputenc}
\usepackage[T1]{fontenc}
\usepackage{amsfonts}
\usepackage{amsmath}
\usepackage{amssymb}
\usepackage{amsthm}
\usepackage{mathrsfs}
\usepackage{amsopn}
\usepackage{indentfirst}
\usepackage{enumerate}
\usepackage{graphicx}

\newtheoremstyle{df}{8pt}{}{}{0cm}{\bf}{.\ }{0pt}{}
\numberwithin{equation}{section}
\newtheorem{tw}{Theorem}[section]
\newtheorem{lm}[tw]{Lemma}
\newtheorem{stw}[tw]{Proposition}
\newtheorem{wn}[tw]{Corollary}
\newtheorem{uw}[tw]{Remark}
\theoremstyle{df}

\newtheorem{ex}[tw]{Example}

\newcommand{\R}{\mathbb{R}}
\newcommand{\C}{\mathbb{C}}
\newcommand{\Ct}{\widetilde{\mathbb{C}}}
\newcommand{\X}{\mathcal{X}}

\newcommand{\const}{\operatorname{const}}

\newcommand{\D}{\operatorname{D}}
\newcommand{\id}{\operatorname{id}}

\newcommand{\DD}{\mathcal{D}}

\newcommand{\Cniesk}{{C^\infty}}

\newcommand{\Jt}{\widetilde{J}}

\newcommand{\pszm}[1]{\frac{\partial}{\partial{#1}}}

\begin{document}
\baselineskip=17pt
\title{On $\widetilde{J}$-tangent affine hyperspheres}
\author{Zuzanna Szancer}
%

\begin{abstract}
    In this paper we study $\Jt$-tangent affine hyperspheres, where $\Jt$ is the canonical para-complex structure on $\R ^{2n+2}$. The main purpose of this paper is to give a classification of $\Jt$-tangent affine hyperspheres of an arbitrary dimension with an involutive distribution $\DD$. In particular, we classify all such hyperspheres in the $3$-dimensional case. We also show that there is a direct relation between $\Jt$-tangent affine hyperspheres and Calabi products. As an application we obtain certain classification results. In particular, we show that, with one exception, all odd dimensional proper flat affine hyperspheres are, after a suitable affine transformation, $\Jt$-tangent. Some examples of $\Jt$-tangent affine hyperspheres are also given.
\end{abstract}

\maketitle

\section{Introduction}
\par Para-complex and paracontact geometry plays an important role in mathematical physics. On the other hand affine differential geometry and in particular affine hyperspheres have been extensively studied over past decades. Some relations between para-complex and affine differential geometry can be found in \cite{LS}, \cite{CLS} and \cite{SZ}.
\par In \cite{SZ2} the author studied $J$-tangent affine hypersurfaces and gave a local classification of $J$-tangent affine hyperspheres with an involutive contact distribution.
\par In this paper we study real affine hyperspheres $f\colon M^{2n+1}\rightarrow \R ^{2n+2}\cong \Ct^{n+1} $ of the para-complex space $\Ct ^{n+1}$ with a $\Jt$-tangent transversal vector field $C$ and an induced almost paracontact structure $(\varphi,\xi,\eta)$. First we show that when $C$ is centro-affine (not necessarily Blaschke) then $f$ can be locally expressed in the form:
\begin{equation}\label{eq::wst}
f(x_1,\ldots,x_{2n},z)=\Jt g(x_1,\ldots,x_{2n})\cosh z-g(x_1,\ldots,x_{2n})\sinh z,
\end{equation}
where $g$ is some smooth immersion defined on an open subset of $\R^{2n}$. Basing on the above result we provide a local classification of all $\Jt$-tangent affine hyperspheres with an involutive distribution $\DD$. We also show that there are no improper $\Jt$-tangent affine hyperspheres.
In particular, using results from \cite{SZ}, we find all $3$-dimensional $\Jt$-tangent affine hyperspheres with the involutive distribution $\DD$. We also give an example of a $\Jt$-tangent affine hypersphere with non-involutive distribution $\DD$.
\par In section 2 we briefly recall the basic formulas of affine differential geometry, we recall the notion of an affine hypersphere and some basic results from para-complex geometry. We also recall the notion of a para-complex affine hypersphere (for details we refer to \cite{SZ}).
\par In section 3 we recall the definitions of an almost paracontact structure introduced for the first time in \cite{KW}. We also recall some elementary results for induced almost paracontact structures that will be used later in this paper.
\par Sections 4 and 5 contain the main results of this paper. In the section 4 we introduce the notion of a $\Jt$-tangent affine hypersphere and prove classification results. In particular, we show that $\Jt$-tangent affine hyperspheres must be proper and there is a strict relation between $\Jt$-tangent affine hyperspheres with the involutive distribution $\DD$ and proper para-complex affine hyperspheres. Finally we show that $\Jt$-tangent affine hyperspheres can be constructed using lower dimensional proper affine hyperspheres. As an application, we classify all $3$-dimensional proper $\Jt$-tangent affine hyperspheres with the involutive distribution $\DD$.
\par In the section 5 we show some applications of the results obtained in section 4. We show that $\Jt$-tangent affine hyperspheres can be classified in terms of Calabi products. Among other we show that (with one exception) all odd dimensional proper flat affine hyperspheres are $\Jt$-tangent affine hyperspheres with
the involutive distribution $\DD$. Moreover, we show that the above mentioned exceptional affine hypersphere is $J$-tangent, where $J$ is the standard complex
structure on $\R^{2n+2}$.

\section{Preliminaries}
We briefly recall the basic formulas of affine differential
geometry. For more details, we refer to \cite{NS}.
\par Let $f\colon M\rightarrow\R^{n+1}$ be an orientable
connected differentiable $n$-dimensional hypersurface immersed in
the affine space $\R^{n+1}$ equipped with its usual flat connection
$\D$. Then for any transversal vector field $C$ we have
\begin{equation}\label{eq::FormulaGaussa}
\D_Xf_\ast Y=f_\ast(\nabla_XY)+h(X,Y)C
\end{equation}
and
\begin{equation}\label{eq::FormulaWeingartena}
\D_XC=-f_\ast(SX)+\tau(X)C,
\end{equation}
where $X,Y$ are vector fields tangent to $M$. It is known that $\nabla$ is a torsion-free connection, $h$ is a symmetric
bilinear form on $M$, called \emph{the second
fundamental form}, $S$ is a tensor of type $(1,1)$, called \emph{the
shape operator}, and $\tau$ is a 1-form, called \emph{the transversal connection form}. Recall that the formula \eqref{eq::FormulaGaussa} is known as the formula of Gauss and the formula \eqref{eq::FormulaWeingartena} is known as the formula of Weingarten.
\par For a hypersurface immersion $f\colon M\rightarrow \R^{n+1}$
a transversal vector field $C$ is said to be \emph{equiaffine}
(resp. \emph{locally equiaffine}) if $\tau=0$ (resp. $d\tau=0$).
For an affine hypersurface $f\colon M\rightarrow \R^{n+1}$ with the transversal vector field $C$ we consider the following volume element on $M$:
$$
\Theta(X_1,\ldots,X_n):=\det[f_\ast X_1,\ldots,f_\ast X_n,C]
$$ for all $X_1,\ldots ,X_n\in \X(M) $. We call $\Theta $ \emph{the induced volume element} on $M$.
Immersion $f\colon M\rightarrow \mathbb{R}^{n+1}$ is said to be a \emph{centro-affine hypersurface} if the position vector $x$ (from origin $o$) for each point $x\in M$ is transversal to the tangent plane of $M$ at $x$. In this case $S=I$ and $\tau=0$.
  If $h$ is nondegenerate (that is $h$ defines a semi-Rie\-man\-nian metric on $M$), we say that the hypersurface or the
hypersurface immersion is \emph{nondegenerate}. In this paper we assume that $f$ is
always nondegenerate.  We have the following
\begin{tw}[\cite{NS}, Fundamental equations]\label{tw::FundamentalEquations}
For an arbitrary transversal vector field $C$ the induced
connection $\nabla$, the second fundamental form $h$, the shape
operator $S$ and the 1-form $\tau$ satisfy
the following equations:
\begin{align}
\label{eq::Gauss}&R(X,Y)Z=h(Y,Z)SX-h(X,Z)SY,\\
\label{eq::Codazzih}&(\nabla_X h)(Y,Z)+\tau(X)h(Y,Z)=(\nabla_Y h)(X,Z)+\tau(Y)h(X,Z),\\
\label{eq::CodazziS}&(\nabla_X S)(Y)-\tau(X)SY=(\nabla_Y S)(X)-\tau(Y)SX,\\
\label{eq::Ricci}&h(X,SY)-h(SX,Y)=2d\tau(X,Y).
\end{align}
\end{tw}
The equations (\ref{eq::Gauss}), (\ref{eq::Codazzih}),
(\ref{eq::CodazziS}), and (\ref{eq::Ricci}) are called the
equations of Gauss, Codazzi for $h$, Codazzi for $S$ and Ricci,
respectively.
\par When $f$ is nondegenerate, there exists a canonical transversal vector field $C$ called \emph{the affine normal field} (or \emph{the Blaschke field}).
The affine normal field is uniquely determined up to sign by the following conditions:
\begin{enumerate}
 \item the metric volume form $\omega_h$ of $h$ is $\nabla$-parallel,
 \item $\omega _h$ coincides with the induced volume form $\Theta$.
\end{enumerate} Recall that  $\omega_h$ is defined by
$$\omega_h(X_1,\ldots,X_n)=|\det[h(X_i,X_j)]|^{1/2},$$ where $\{X_1,\ldots ,X_n\}$ is any positively oriented basis relative to the induced volume form $\Theta$. The affine immersion $f$ with a Blaschke field $C$ is called \emph{ a Blaschke hypersurface}.
In this case fundamental equations can be rewritten as follows
\begin{tw}[\cite{NS}, Fundamental equations]\label{tw::FundamentalEquationsBlaschke}
For a Blaschke hypersurface $f$, we have the following fundamental equations:
\begin{align*}
&R(X,Y)Z=h(Y,Z)SX-h(X,Z)SY,\\
&(\nabla_X h)(Y,Z)=(\nabla_Y h)(X,Z),\\
&(\nabla_X S)(Y)=(\nabla_Y S)(X),\\
&h(X,SY)=h(SX,Y).
\end{align*}
\end{tw}
A Blaschke hypersurface is called \emph{an affine hypersphere} if $S=\lambda I$, where $\lambda =\const.$\\
If $\lambda =0$ $f$ is called \emph{an improper affine hypersphere}, if $\lambda \neq 0$ a hypersurface $f$ is called \emph{a proper affine hypersphere.}
\par Now, we will recall a notion of para-complex affine hypersurfaces, for details we refer to \cite{SZ}. More information on para-complex geometry one may found for example in \cite{CFG} and \cite{AB}.
\par Let $g\colon M^{2n}\rightarrow\R^{2n+2}$ be an immersion and let $\Jt$ be the standard para-complex structure on $\R^{2n+2}$. That is
$$
\Jt(x_1,\ldots,x_{n+1},y_1,\ldots,y_{n+1}):=(y_1,\ldots,y_{n+1},x_1,\ldots,x_{n+1}).
$$
We always identify $(\R^{2n+2},\Jt)$ with $\Ct^{n+1}$.
\par Assume now that $g_\ast(TM)$ is $\Jt$-invariant and  $\Jt|_{g_\ast(T_xM)}$ is a para-complex structure on $g_\ast(T_xM)$ for every $x\in M$.
Then $\Jt$ induces an almost para-complex structure on $M$, which we will also denote by $\Jt$. Moreover, since $(\R^{2n+2},\Jt)$ is para-complex then $(M,\Jt)$ is para-complex as well.
By assumption we have that $dg\circ \Jt=\Jt\circ dg$ that is $g\colon M^{2n}\rightarrow\R^{2n+2}\cong\Ct^{n+1}$ is a para-holomorphic immersion. Since para-complex dimension of $M$ is $n$, immersion $g$
is called a \emph{para-holomorphic hypersurface}.
\par Let $g\colon M^{2n}\rightarrow\R^{2n+2}$ be an affine hypersurface of codimension 2 with a transversal bundle $\mathcal{N}$. If $g$ is para-holomorphic then it is  called
 \emph{affine para-holomorphic hypersurface}. If additionally the transversal bundle $\mathcal{N}$ is $\Jt$-in\-var\-iant then $g$ is called a \emph{para-complex affine hypersurface}.
\par Let $g\colon M^{2n}\rightarrow \R^{2n+2}$ be a para-holomorphic hypersurface. We say that $g$ is \emph{para-complex centro-affine hypersurface}  if $\{g,\Jt g\}$ is a transversal bundle for $g$.
\begin{tw}[\cite{SZ}]
Let $g\colon M^{2n}\rightarrow \R^{2n+2}$ be a para-holomorphic hypersurface. Then for every $x\in M$ there exists a neighborhood $U$ of $x$ and
a transversal vector field $\zeta\colon U\rightarrow\R^{2n+2}$ such that $\{\zeta,\Jt\zeta\}$ is a transversal bundle for $g|_U$. That is $g|_U$ considered with $\{\zeta,\Jt\zeta\}$ is a para-complex affine hypersurface.
\end{tw}
Now let $g\colon M^{2n}\rightarrow \R^{2n+2}$ be a para-holomorphic hypersurface and let $\zeta\colon U\rightarrow\R^{2n+2}$ be a local transversal vector field on $U\subset M$ such that
$\{\zeta,\Jt\zeta\}$ is a transversal bundle to $g$. For all tangent vector fields $X,Y\in\X(U)$ we can decompose $D_XY$ and $D_X\zeta$ into tangent and transversal part. Namely, we have
\begin{align*}
\D_Xg_{\ast}Y&=g_{\ast}(\nabla _XY)+h_1(X,Y)\zeta +h_2(X,Y)\Jt\zeta \quad\textrm{(formula of Gauss)}, \\
\D_X\zeta &=-g_{\ast}(SX)+\tau _1(X)\zeta +\tau _2(X)\Jt\zeta \quad\textrm{(formula of Weingarten)},
\end{align*}
where $\nabla$ is a torsion free affine connection on $U$, $h_1$ and $h _2$ are symmetric bilinear forms on $U$, $S$ is a $(1,1)$-tensor field on $U$ and $\tau _1$ and $\tau _2$ are $1$-forms on $U$.
We have the following relations between  $h_1$ and $h_2$
\begin{lm}[\cite{LS},\cite {SZ}]\label{lm:oh1h2}
\begin{align}
\label{eq::h1h2}h_1(X,\Jt Y)&=h_1(\Jt X,Y)=h_2(X,Y),\\
\label{eq::h2h1}h_2(X,\Jt Y)&=h_1(X,Y).
\end{align}
\end{lm}
\par On $U$ we define the volume form $\theta _{\zeta}$ by the formula
$$
\theta _{\zeta}(X_1,\ldots ,X_{2n}):=\det (g_\ast X_1,\ldots ,g_\ast X_{2n},\zeta ,\Jt \zeta)
$$
for tangent vectors $X_i$, $i=1,\ldots,2n$. Let us consider the function $H_{\zeta}$ on $U$ defined by
$$
H_{\zeta}:=\det[h_1(X_i,X_j)]_{i,j=1\ldots 2n},
$$
where $X_1, \ldots ,X_{2n}$ is a local basis on $TU$ such that $\theta _{\zeta}(X_1,\ldots ,X_{2n})=1$. This definition is independent of the choice of basis.
We say that a hypersurface is \emph{nondegenerate} if $h_1$ (and in consequence $h_2$) is nondegenerate.
\par When $g$ is nondegenerate there exist transversal vector fields $\zeta$ satisfying the following two conditions:
\begin{align*}
|H_{\zeta}|=1,\\
\tau _1=0.
\end{align*}
Such vector fields are called \emph{affine normal vector fields}. 
 In \cite{SZ} we showed that on every para-holomorphic hypersurface we may find (at least locally) an affine normal vector field.
\par A nondegenerate para-complex hypersurface is said to be a \emph{proper para-complex affine hypersphere} if there exists an affine normal vector field $\zeta$ such
that $S=\alpha I$, where $\alpha\in \R \setminus \{0\}$ and $\tau _2=0$. If there exists an affine normal vector field $\zeta$ such that $S=0$
and $\tau _2=0$ we say about an \emph{improper para-complex affine hypersphere}.
Note that the above definition is very analogous to the definition of complex affine hypersphere introduced by F. Dillen, L. Vrancken and L. Verstraelen in \cite{DVV}.

\section{Almost paracontact structures}
\par Let $\dim M=2n+1$ and $f\colon M\rightarrow \R^{2n+2}$ be
a nondegenerate (relative to the second fundamental form) affine hypersurface. We always assume that $\R^{2n+2}$ is endowed with
the standard para-complex structure $\widetilde{J}$
$$
\widetilde{J}(x_1,\ldots,x_{n+1},y_1,\ldots,y_{n+1})=(y_1,\ldots,y_{n+1},x_1,\ldots,x_{n+1}).
$$
Let $C$ be a transversal vector field on $M$. We say that $C$ is
\emph{$\widetilde{J}$-tangent} if $\widetilde{J}C_x\in f_\ast(T_xM)$ for every $x\in M$.
We also define a distribution $\DD$ on $M$ as the biggest $\widetilde{J}$-invariant distribution on $M$, that is
$$
\DD_x=f_\ast^{-1}(f_\ast(T_xM)\cap \widetilde{J}(f_\ast(T_xM)))
$$
for every $x\in M$. We have that $\dim\DD _x\geq 2n$. If for some $x$ the $\dim\DD _x=2n+1$ then $\DD _x=T_xM$ and it is not possible to find a $\widetilde{J}$-tangent transversal vector field in a neighbourhood of $x$. Since we only study hypersurfaces with a $\widetilde{J}$-tangent transversal vector field, then we always have $\dim\DD=2n$. The distribution $\DD$ is smooth as an intersection of two smooth distributions and because $\dim \DD$ is constant.  A vector field $X$ is called a \emph{$\DD$-field}
if $X_x\in\DD_x$ for every $x\in M$. We use the notation $X\in\DD$ for vectors as well as for $\DD$-fields.
We say that the distribution $\DD$ is nondegenerate if
$h$ is nondegenerate on $\DD$.
\par A $(2n+1)$-dimensional manifold $M$ is said to have an
\emph{almost paracontact structure} if there exist on $M$ a tensor
field $\varphi$ of type (1,1), a vector field $\xi$ and a 1-form
$\eta$ which satisfy
\begin{align}
\varphi^2(X)&=X-\eta(X)\xi,\\
\eta(\xi)&=1
\end{align}
for every $X\in TM$
and the tensor field $\varphi$ induces an almost para-complex structure on the distribution $\DD=\operatorname{ker}\eta$. That is the eigendistributions $\DD ^{+},\DD ^{-}$ corresponding to the eigenvalues $1,-1$
of $\varphi$ have equal dimension $n$.
\par Let $f\colon M\rightarrow \R^{2n+2}$ be a nondegenerate affine
hypersurface with a $\widetilde{J}$-tangent transversal vector field $C$. Then
we can define a vector field $\xi$, a 1-form $\eta$ and a tensor field
$\varphi$ of type (1,1) as follows:
\begin{align}
&\xi:=\widetilde{J}C;\\
\label{etanaD::eq::0}&\eta|_\DD=0 \text{ and } \eta(\xi)=1; \\
&\varphi|_\DD=\widetilde{J}|_\DD \text{ and } \varphi(\xi)=0.
\end{align}
It is easy to see that $(\varphi,\xi,\eta)$ is an almost paracontact
structure on $M$. This structure is called the \emph{induced almost
paracontact structure}.
For an induced almost paracontact structure we have the following theorem
\begin{tw}[\cite{SZZ}]\label{tw::Wzory}
Let $f\colon M\rightarrow \mathbb{R}^{2n+2}$ be an affine hypersurface with a $\widetilde{J}$-tangent transversal vector field $C$.
If $(\varphi,\xi,\eta)$ is an induced almost paracontact structure on $M$
then the following equations hold:
\begin{align}
\label{Wzory::eq::1}&\eta(\nabla_XY)=h(X,\varphi Y)+X(\eta(Y))+\eta(Y)\tau(X),\\
\label{Wzory::eq::2}&\varphi(\nabla_XY)=\nabla_X\varphi Y-\eta(Y)SX-h(X,Y)\xi,\\
\label{Wzory::eq::3}&\eta([X,Y])=h(X,\varphi Y)-h(Y,\varphi X)+X(\eta(Y))-Y(\eta(X))\\
\nonumber &\qquad\qquad\quad+\eta(Y)\tau(X)-\eta(X)\tau(Y),\\
\label{Wzory::eq::4}&\varphi([X,Y])=\nabla_X\varphi Y-\nabla_Y\varphi X+\eta(X)SY-\eta(Y)SX,\\
\label{Wzory::eq::5}&\eta(\nabla_X\xi)=\tau(X),\\
\label{Wzory::eq::6}&\eta(SX)=-h(X,\xi)
\end{align}
for every $X,Y\in \X(M)$.
\end{tw}

\section{$\widetilde{J}$-tangent affine hyperspheres}
An affine hypersphere with a transversal $\Jt$-tangent Blaschke field we call \emph{a $\Jt$-tangent affine hypersphere}. We start this section with the following useful lemma related to differential equations
\begin{lm}\label{lm::Differential-Equation}
Let $F\colon I\rightarrow\R^{2n}$ be a smooth function on an interval $I$.
If $F$ satisfies the differential equation
\begin{equation}\label{eq::RR}
F'(z)=-\widetilde{J}F(z),
\end{equation}
then $F$ is of the form
\begin{equation}\label{eq::RR-solution}
F(z)=\widetilde{J}v\cosh z-v\sinh z,
\end{equation}
where $v\in\R^{2n}$.
\end{lm}
\begin{proof}
It is not difficult to check that functions of the form (\ref{eq::RR-solution}) satisfy the differential equation (\ref{eq::RR}).
On the other hand, since (\ref{eq::RR}) is a first-order ordinary differential equation, the Picard-Lindel\"{o}f theorem implies that any solution of
(\ref{eq::RR}) must be of the form (\ref{eq::RR-solution}).
\end{proof}
Using the above lemma, we can prove the following theorem
\begin{tw}\label{tw::local-form-of-centroaffine}
Let $f\colon M\rightarrow\R^{2n+2}$ be a centro-affine hypersurface with a $\widetilde{J}$-tangent centro-affine vector field. Then there exist an open subset $U\subset\R^{2n}$,
an interval $I\subset\R$ and an immersion $g\colon U\rightarrow\R^{2n+2}$ such that $f$ can be locally expressed in the form
\begin{equation}\label{eq::local-form-of-centroaffine}
f(x_1,\ldots ,x_{2n},z)=\widetilde{J}g(x_1,\ldots ,x_{2n})\cosh z-g(x_1,\ldots ,x_{2n})\sinh z
\end{equation}
for all $(x_1,\ldots,x_{2n},z)\in U\times I$.
\end{tw}
\begin{proof}
Denote $C:=-f$. Since $f$ is a centro-affine hypersurface with a $\widetilde{J}$-tangent transversal vector field then we have $\widetilde{J}C=-\widetilde{J}f\in f_{\ast}(TM)$.
Therefore, for every $x\in M$, there exists a neighborhood $V$ of $x$ and a map $\psi(x_1,\ldots,x_{2n},z)$ on $V$ such that
$$
f_{\ast}\pszm{z}=\widetilde{J}C.
$$
That is $f$ can be locally expressed in the form $f(x_1,\ldots ,x_{2n},z),$ where $f_z=-\widetilde{J}f$. Now using Lemma \ref{lm::Differential-Equation} we obtain the thesis.
\end{proof}

\par When the distribution $\DD$ is involutive we have

\begin{tw}\label{tw::Ivolutivity-of-D-for-centroaffine}
Let $f\colon M\rightarrow \R^{2n+2}$ be an affine hypersurface with a centro-affine $\Jt$-tangent vector field $C=-\overrightarrow{of}$.
If the distribution $\DD$ is involutive then for every $x\in M$ there exists a para-complex centro-affine immersion $g\colon V\rightarrow \R^{2n+2}$ defined on
an open subset $V\subset \R^{2n}$ such that $f$ can be expressed in the neighborhood of $x$ in the form
\begin{equation}\label{eq::f}
f(x_1,\ldots,x_{2n},z)=\Jt g(x_1,\ldots,x_{2n})\cosh z-g(x_1,\ldots,x_{2n})\sinh z.
\end{equation}
Moreover, if $g\colon V\rightarrow \R^{2n+2}$ is a para-complex centro-affine immersion then $f$ given by the formula {\rm(\ref{eq::f})} is an affine hypersurface with a centro-affine $\Jt$-tangent vector field and an involutive distribution $\DD$.
\end{tw}
\begin{proof}
Let $(\varphi ,\xi ,\eta)$ be an induced almost paracontact structure on $M$ induced by $C$.
The Frobenius theorem implies that for every $x\in M$ there exist an open neighborhood $U\subset M$ of $x$ and linearly independent vector fields $X_1,\ldots,X_{2n},X_{2n+1}=\xi\in\X(U)$
such that $[X_i,X_j]=0$ for $i,j=1,\ldots,2n+1$.
For every $i=1,\ldots,2n$ we have $X_i=D_i+\alpha_i\xi$ where $D_i\in\DD$ and $\alpha_i\in\Cniesk(U)$. Thus we have
$$
0=[X_i,\xi]=[D_i,\xi]-\xi(\alpha_i)\xi.
$$
Now (\ref{Wzory::eq::3}) and (\ref{Wzory::eq::6}) imply that $[D_i,\xi]$ and $\xi(\alpha_i)=0$. We also have
$$
0=[X_i,X_j]=[D_i,D_j]-D_j(\alpha_i)\xi+D_i(\alpha_j)\xi
$$
for $i=1,\ldots,2n$. Since $\DD$ is involutive the above equality implies that $[D_i,D_j]=0$ for $i,j=1,\ldots,2n$. Of course the vector fields
$D_1,\ldots,D_{2n},\xi$ are linearly independent, so there exists a map $\psi(x_1,\ldots,x_{2n},z)$ on $U$ such that
$$
\pszm{z}=\xi,\quad \pszm{x_i}=D_i,\quad \text{$i=1,\ldots,2n.$}
$$
Now applying Lemma \ref{lm::Differential-Equation} we find that $f$ can be locally expressed in the form
$$
f(x_1,\ldots ,x_{2n},z)=\widetilde{J}g(x_1,\ldots ,x_{2n})\cosh z-g(x_1,\ldots ,x_{2n})\sinh z,
$$
where $g\colon V\rightarrow \R^{2n+2}$ is an immersion defined on an open subset $V\subset \R^{2n}$.
Moreover, since $\pszm{x_i}\in\DD$ we have that
$$
f_{x_i}=\widetilde{J}g_{x_i}\cosh z - g_{x_i}\sinh z\in f_{\ast}(D).
$$
Since $f_{\ast}(D)$ is $\Jt$-invariant we also have
$$
\Jt f_{x_i}=g_{x_i}\cosh z - \Jt g_{x_i}\sinh z \in f_{\ast}(D).
$$
The above implies that $g_{x_i}\in f_\ast(D)$ for $i=1,\ldots,2n$.
Since $\{g_{x_i}\}$ are linearly independent, they form a basis of $f_{\ast}(D)$ (note that $\dim f_{\ast}(D)=2n$) i. e.
$$
f_{\ast}(D)=\operatorname{span}\{g_{x_1},\ldots,g_{x_{2n}}\}.
$$
Since $f_{\ast}(D)$ is $\Jt$-invariant  we also have that
$$
\Jt g_{x_i}\in f_{\ast}(D)=\operatorname{span}\{g_{x_1},\ldots,g_{x_{2n}}\}.
$$
That is, $\Jt g_{x_i}=\sum\alpha_ig_{x_i}$, where $\alpha_i\in\Cniesk(U)$. Since $g$ does not depend on variable $z$, the functions $\alpha_i$ also do not, thus $\alpha_i\in\Cniesk(V)$.
\par In this way we have shown that for $g\colon V\rightarrow \R^{2n+2}$ the tangent space $TV$ is $\Jt$-invariant (we can transfer $\Jt$ from $g_\ast(TV)$ to $TV$). Since
$\Jt|_{f_\ast(D)}$ is para-complex and $f_\ast(D)=\operatorname{span}_{\Cniesk(U)}\{g_{x_1},\ldots,g_{x_{2n}}\}$, $\Jt$ is a para-complex structure on $TV$.
Finally $g$ is para-holomorphic. Since $f$ is an immersion, $\{g_{x_1},\ldots,g_{x_{2n}},\Jt g\}$ are linearly independent. Moreover, because $f$ is centro-affine, we also have that $g$
is linearly independent with $\{g_{x_1},\ldots,g_{x_{2n}},\Jt g\}$. That is $\{g,\Jt g\}$ is a $\Jt$-invariant transversal bundle for $g_\ast(TV)$ and, in consequence, $g$ is a para-complex affine immersion.
\par In order to prove the second part of the theorem, note that since $g$ is a centro-affine para-complex affine immersion, then
$\{f_{x_1},\ldots,f_{x_{2n}},f_{z},f\}$ are linearly independent. It means that $f$ is an immersion and is centro-affine. Moreover, $f$ is $\Jt$-tangent since
$\Jt(-\overrightarrow{of})=-g\cosh z +\Jt g\sinh z=f_z$. In particular, $g$ is para-holomorphic. That is, we have $\Jt g_{x_i}=\sum_{j=1}^{2n}\alpha_{ij}g_{x_j}$ for $i=1,\ldots,2n$.
Now, by straightforward computations we get $\sum_{j=1}^{2n}\alpha_{ij}f_{x_j}=\Jt f_{x_i}$ for $i=1,\ldots,2n$. That is, $\Jt f_{x_i}\in\operatorname{span}\{f_{x_1},\ldots,f_{x_{2n}}\}$.
In this way we have shown that $\operatorname{span}\{f_{x_1},\ldots,f_{x_{2n}}\}$ is $\Jt$-invariant and since its dimension is $2n$ it must be equal to $f_\ast(D)$.
Now it is easy to see that $\DD=\{\pszm{x_1},\ldots,\pszm{x_{2n}}\}$ is involutive as generated by the canonical vector fields.
\end{proof}
For $\Jt$-tangent affine hyperspheres we have the following classification theorems:
\begin{tw}\label{tw::NoImpropSph}
There are no improper $\widetilde{J}$-tangent affine hyperspheres.
\end{tw}
\begin{proof}
By (\ref{Wzory::eq::6}) we have $\eta (SX)=-h(X,\xi)$ for all $X\in \X(M)$. Since $S=0$ we have $h(X,\xi)=0$ for every $X\in \X(M)$, which contradicts nondegeneracy of $h$.
\end{proof}


\begin{tw}\label{tw::ParacomplexHyperspheresWithInvolutiveDistribution}
Let $f\colon M\rightarrow \R ^{2n+2}$ be a $\Jt$-tangent affine hypersphere with an involutive distribution $\DD$. Then $f$ can be locally expressed in the form:
\begin{align}
\label{eq::Jhypersphere} f(x_1,\ldots ,x_{2n},z)=\Jt g(x_1,\ldots ,x_{2n})\cosh z-g(x_1,\ldots ,x_{2n})\sinh z,
\end{align}
where $g$ is a proper para-complex affine hypersphere. Moreover, the converse is also true in the sense that if $g$ is a proper para-complex affine hypersphere then $f$ given by the formula {\rm(\ref{eq::Jhypersphere})} is a $\Jt$-tangent affine hypersphere with an involutive distribution $\DD$.
\end{tw}
\begin{proof}
($\Rightarrow$) First note that due to Theorem \ref{tw::NoImpropSph} $f$ must be a proper affine hypersphere. Let $C$ be a $\Jt$-tangent affine normal field. There exists $\lambda \in \mathbb{R}\setminus \{0\}$ such that $C=-\lambda f$. Since $C$ is $\Jt$-tangent and transversal the same is $\frac{1}{\lambda}C=-f$. That is, $f$ satisfies assumptions of Theorem \ref{tw::Ivolutivity-of-D-for-centroaffine}. By Theorem \ref{tw::Ivolutivity-of-D-for-centroaffine}
there exists a para-complex centro-affine immersion $g\colon V\rightarrow \R^{2n+2}$ defined on
an open subset $V\subset \R^{2n}$ and there exists an open interval $I$ such that $f$ can be locally expressed in the form
\begin{equation}
f(x_1,\ldots,x_{2n},z)=\Jt g(x_1,\ldots,x_{2n})\cosh z-g(x_1,\ldots,x_{2n})\sinh z
\end{equation}
for $(x_1,\ldots ,x_{2n})\in V$ and $z\in I$.
\par Let $\zeta :=-|\lambda|^{\frac{2n+3}{2n+4}}g$. Bundle $\{\zeta, \Jt \zeta\}$ is transversal to $g$, because $g$ is para-complex centro-affine. Let $\nabla, h_1, h_2, S, \tau_1, \tau_2$ be induced objects on $V$ by $\zeta$.
Using the Weingarten formula for $g$ and $\zeta$ we get
$$
\D_{\partial _{x_i}}\zeta =-g_{\ast}(S\partial _{x_i})+\tau _1(\partial _{x_i})\zeta +\tau _2(\partial _{x_i})J\zeta .
$$
On the other hand, by straightforward computations we have
$$
\D_{\partial _{x_i}}\zeta=\partial _{x_i}(\zeta)=-|\lambda|^{\frac{2n+3}{2n+4}}g_{\ast}(\partial _{x_i}).
$$
Thus, we obtain
\begin{align}
\label{eq::Stau1tau2} S=|\lambda|^{\frac{2n+3}{2n+4}}I,\quad \tau _1=0, \quad \tau_2=0.
\end{align}
Now, it is enough to show that $\zeta$ is an affine normal vector field that is $|H_{\zeta}|=1.$
Since $g$ is para-holomorphic, without loss of generality, we may assume that
$$
\partial_{x_{n+i}}=\Jt\partial _{x_i}
$$
for $i=1,\ldots, n.$ Let $h$ be the second fundamental form for $f$.

Using similar methods like in the proof of Theorem 4.1 from \cite{SZ2} one may compute


$$
-\lambda h(\partial _{x_i},\partial _{x_j})=-|\lambda|^{\frac{2n+3}{2n+4}}h_1(\partial _{x_i},\partial _{x_j}),\quad h(\partial _z,\partial _z)=-\frac{1}{\lambda}
$$
and
$$
h(\partial _z,\partial _{x_i})=h(\partial _{x_i},\partial _z)=0
$$
for $i,j=1,\ldots, 2n.$ Let us denote
$$
a:=\theta _{\zeta}(\partial _{x_1},\ldots ,\partial _{x_n},\Jt\partial _{x_1},\ldots ,\Jt\partial _{x_n}).
$$
Then we have
\begin{align*}
\det h:&=\det
\left[
\begin{matrix}
h(\partial _{x_1},\partial _{x_1}) & h(\partial _{x_1},\partial _{x_2}) & \cdots & h(\partial _{x_1},\partial _{x_{2n}}) & 0 \\
h(\partial _{x_2},\partial _{x_1}) & h(\partial _{x_2},\partial _{x_2}) & \cdots & h(\partial _{x_2},\partial _{x_{2n}}) & 0 \\
\vdots & \vdots & \ddots & \vdots & \vdots  \\
h(\partial _{x_{2n}},\partial _{x_1}) & h(\partial _{x_{2n}},\partial _{x_2}) & \cdots & h(\partial _{x_{2n}},\partial _{x_{2n}}) & 0 \\
    0   &    0  &    \cdots & 0 & -\frac{1}{\lambda}
\end{matrix}\right]\\
&=-\frac{1}{\lambda}\det [h(\partial _{x_i},\partial _{x_j})]=-\frac{1}{\lambda}\cdot (\frac{1}{\lambda}\cdot |\lambda|^{\frac{2n+3}{2n+4}} )^{2n}\det [h_1(\partial _{x_i},\partial _{x_j})]\\
&=-\frac{1}{\lambda}\cdot |\lambda|^{-\frac{2n}{2n+4}}a^2 H_\zeta
\end{align*}
and
\begin{align}\label{eq::OmegaH2}
(\omega_h)^2=|\det h|=|\lambda|^{\frac{-2n-2}{n+2}}a^2|H_\zeta|.
\end{align}
Again by similar computation like in \cite{SZ2} we get

\begin{align*}
\omega _h
&=-\lambda \cdot (|\lambda|^{-\frac{2n+3}{n+2}})\cdot (-1)^{n+1}\theta _{\zeta}(\partial _{x_1},\ldots ,\partial _{x_{2n}})=(-1)^{n+2}\cdot \lambda \cdot (|\lambda|^{-\frac{2n+3}{n+2}})a.
\end{align*}
Using the above formula in (\ref{eq::OmegaH2}) we easily obtain
$$
|H_\zeta|=a^{-2}|\lambda|^{\frac{2n+2}{n+2}}\cdot \lambda^2 \cdot |\lambda|^{-\frac{4n+6}{n+2}}a^2=1.
$$

($\Leftarrow $) Let $g\colon U \rightarrow \mathbb{R}^{2n+2}$ be a proper para-complex affine hypersphere.
Since $g$ is a proper para-complex affine hypersphere there exists $\alpha \neq 0$ such that $\zeta =-\alpha g$ is an affine normal vector field. Without loss of generality we may assume that $\alpha >0$. Since both, $g$ and $\Jt g$ are transversal, we see that $\{g_{x_1},\ldots ,g_{x_{2n}},g,\Jt g\}$ forms the basis of $\mathbb{R}^{2n+2}$. The above implies that $$f\colon U\times I\ni (x_1,\ldots ,x_{2n},z)\mapsto f(x_1,\ldots ,x_{2n},z)\in \mathbb{R}^{2n+2}$$ given by the formula:
$$
f(x_1,\ldots ,x_{2n},z):=\Jt g(x_1,\ldots ,x_{2n})\cosh z-g(x_1,\ldots ,x_{2n})\sinh z
$$ is an immersion and $C:=-\alpha ^{\frac{2n+4}{2n+3}}\cdot f$ is a transversal vector field.
The field $C$ is $\Jt$-tangent because $\Jt C=\alpha ^{\frac{2n+4}{2n+3}}f_z$. Since $C$ is equiaffine and $S=\alpha ^{\frac{2n+4}{2n+3}}I$ it is enough to show that $\omega _h=\theta $ for some positively oriented (relative to $\theta$) basis on $U\times I$. Let $\partial _{x_1},\ldots ,\partial _{x_{2n}},\partial _z$ be a local coordinate system on $U\times I$. Since $g$ is para-holomorphic we may assume that $\partial _{x_{n+i}}=\Jt\partial _{x_i}$ for $i=1,\ldots, n$.

Then we have
\begin{align*}
\theta (\partial _{x_1},&\ldots ,\partial _{x_{2n}},\partial _z)
=-\alpha ^{-\frac{2n+2}{2n+3}}\cdot (-1)^{n+1}\theta _{\zeta}(\partial _{x_1},\ldots ,\partial _{x_{2n}}).
\end{align*}
That is
\begin{align}\label{eq::tetazeta}
\theta _{\zeta}(\partial _{x_1},\ldots ,\partial _{x_{2n}})=(-1)^{n}\alpha ^{\frac{2n+2}{2n+3}}\cdot \theta (\partial _{x_1},&\ldots ,\partial _{x_{2n}},\partial _z).
\end{align}
In a similar way as in the proof of the first implication we compute
\begin{align*}
\det h&=-\alpha ^{-\frac{2n+4}{2n+3}}\cdot \Big(\frac{\alpha}{\alpha ^{\frac{2n+4}{2n+3}}}\Big)^{2n}\det h_1\\
&=-\alpha ^{-\frac{2n+4}{2n+3}}\cdot \alpha ^{-\frac{2n}{2n+3}}\det h_1\\
&=-\alpha ^{\frac{-4n-4}{2n+3}}\det h_1.
\end{align*}
The above implies that
$$
(\omega_h)^2=|\det h|=\alpha ^{\frac{-4n-4}{2n+3}}|\det h_1|.
$$
Since
$$
|\det h_1|=|H_{\zeta}|[\theta _{\zeta}(\partial _{x_1},\ldots ,\partial _{x_{2n}})]^2,
$$
we obtain
$$
(\omega_h)^2=\alpha ^{\frac{-4n-4}{2n+3}}|H_{\zeta}|[\theta _{\zeta}(\partial _{x_1},\ldots ,\partial _{x_{2n}})]^2.
$$
Finally, using the fact that $|H_{\zeta}|=1$ and (\ref{eq::tetazeta}), we get $$\omega _h=|\theta (\partial _{x_1},\ldots ,\partial _{x_{2n}},\partial _z)|.$$ The proof is completed.
\end{proof}

Immediately from the proof of the above theorem we get
\begin{wn}
If $f$ is a $\Jt$-tangent affine hypersphere with the shape operator $S=\lambda \id$ and $g$ is a para-complex affine hypersphere (related to $f$) with the shape operator $\widetilde{S}=\alpha \id$ then $\lambda$ and $\alpha$ are related by the following formula:
$$|\lambda|=|\alpha|^{\frac{2n+4}{2n+3}}.$$
\end{wn}
Now, we shall recall a classification theorem for para-complex affine hyperspheres.
\begin{tw}[\cite{SZ}]\label{th::og}
Let $g\colon M\rightarrow \R^{2n+2}$ be a para-complex affine hypersphere with a transversal bundle $\{\zeta ,\Jt \zeta\}$. Then there exist open subsets $U_1\subset \R^{n}$, $U_2\subset \R^{n}$ and (real) affine hyperspheres
$$
f_1\colon U_1\rightarrow \R^{n+1}, \quad f_2\colon U_2\rightarrow \R^{n+1}
$$
such that $g$ can be locally expressed in the form
\begin{align}\label{f::og}
g=f_1\times f_2 +\Jt \circ (f_1\times (-f_2)).
\end{align}
Moreover, if $g$ is proper (respectively improper) then both $f_1$ and $f_2$ are proper (respectively improper) as well. The converse is also true, in the sense, that for every two proper (respectively improper) $n$-dimensional affine hyperspheres $f_1$ and $f_2$ the formula {\rm(\ref{f::og})} defines a proper (respectively improper) para-complex affine hypersphere.
\end{tw}
The following theorem allows us to construct $\Jt$-tangent affine hyperspheres using standard proper affine hyperspheres. Namely we have
\begin{tw}\label{tw::ParaComplexSpheresDecomposition}
Let $f\colon M\rightarrow \R ^{2n+2}$ be a $\Jt$-tangent affine hypersphere with an involutive distribution $\DD$. Then $f$ can be locally expressed in the form:
\begin{align}\label{eq::SphereDecomposition}
f(x_1,\ldots ,x_{n},y_1,\ldots ,y_n,z)=\\ \nonumber &\hspace{-5cm}\Big(\Jt \circ(f_1\times f_2)+f_1\times(-f_2)\Big)(x_1,\ldots ,x_{n},y_1,\ldots ,y_n)\cosh z\\ \nonumber&\hspace{-5cm}-\Big((f_1\times f_2)+
\Jt \circ (f_1\times(-f_2))\Big)(x_1,\ldots ,x_{n},y_1,\ldots ,y_n)\sinh z,
\end{align}
where $f_1$ and $f_2$ are proper $n$-dimensional affine hyperspheres. Moreover, the converse is also true in the sense that if $f_1$ and $f_2$ are proper $n$-dimensional affine hyperspheres then $f$ given by the above formula is a proper $\Jt$-tangent affine hypersphere with an involutive distribution $\DD$.
\end{tw}
\begin{proof}
The proof is an immediate consequence of Theorem \ref{tw::ParacomplexHyperspheresWithInvolutiveDistribution} and Theorem \ref{th::og}.
\end{proof}

Since the only $1$-dimensional proper affine spheres are the ellipse and hyperbola
we can obtain the complete local classification of $3$-dimensional $\Jt$-tangent affine hyperspheres with an involutive distribution $\DD$. Namely we have
\begin{tw}\label{tw::klasyfikacjaSfer3dim}
Let $f\colon M^3\rightarrow\R^{4}$ be a $\Jt$-tangent affine hypersphere with an involutive distribution $\DD$. Then up to $\Jt$-invariant affine transformation
$f$ is locally equivalent to one of the following hypersurfaces:
\begin{align}
\label{eq::3dim::1}
f_1(x,y,z)&=
\left(
\begin{matrix}
\cos x+\cos y\\
\sin x+\sin y\\
\cos x-\cos y\\
\sin x-\sin y
\end{matrix}\right)\cosh z
-\left(
\begin{matrix}
\cos x-\cos y\\
\sin x-\sin y\\
\cos x+\cos y\\
\sin x+\sin y
\end{matrix}\right)\sinh z, \\
\label{eq::3dim::2} f_2(x,y,z)&=
\left(
\begin{matrix}
\cosh x+\cosh y\\
\sinh x+\sinh y\\
\cosh x-\cosh y\\
\sinh x-\sinh y
\end{matrix}\right)\cosh z
-\left(
\begin{matrix}
\cosh x-\cosh y\\
\sinh x-\sinh y\\
\cosh x+\cosh y\\
\sinh x+\sinh y
\end{matrix}\right)\sinh z,\\
\label{eq::3dim::3}
f_3(x,y,z)&=
\left(
\begin{matrix}
\cos x+\cosh y\\
\sin x+\sinh y\\
\cos x-\cosh y\\
\sin x-\sinh y
\end{matrix}\right)\cosh z
-\left(
\begin{matrix}
\cos x-\cosh y\\
\sin x-\sinh y\\
\cos x+\cosh y\\
\sin x+\sinh y
\end{matrix}\right)\sinh z, \\
\label{eq::3dim::4} f_4(x,y,z)&=
\left(
\begin{matrix}
\cosh x+\cos y\\
\sinh x+\sin y\\
\cosh x-\cos y\\
\sinh x-\sin y
\end{matrix}\right)\cosh z
-\left(
\begin{matrix}
\cosh x-\cos y\\
\sinh x-\sin y\\
\cosh x+\cos y\\
\sinh x+\sin y
\end{matrix}\right)\sinh z.
\end{align}
\end{tw}
\begin{proof}
By Theorem \ref{tw::ParaComplexSpheresDecomposition} $f$ can be locally obtained from $1$-dimensional proper affine spheres $f_1$ and $f_2$.
Since the only $1$-dimensional proper affine spheres are the ellipse and hyperbola then $f_i$ is affinely equivalent to either
$\gamma_1(t)=(\cos t,\sin t)$ or $\gamma_2(t)=(\cosh t,\sinh t)$. Now one may find affine transformations $P,Q$ of $\R^2$ such that
\begin{align}\label{eq::f1f2}
f_1=P\circ\gamma_{i_0} \quad\text{and}\quad f_2=Q\circ\gamma_{j_0}
\end{align}
for some $i_0,j_0\in\{1,2\}$.
Applying (\ref{eq::f1f2}) to (\ref{eq::SphereDecomposition}) we get
\begin{align*}
f(x,y,z)=
A\circ
\Bigg[\Big(\Jt \circ(\gamma_{i_0}\times \gamma_{j_0})+\gamma_{i_0}\times(-\gamma_{j_0})\Big)(x,y)\cosh z\\
-\Big((\gamma_{i_0}\times \gamma_{j_0})+\Jt \circ (\gamma_{i_0}\times(-\gamma_{j_0}))\Big)(x,y)\sinh z\Bigg]\\
=A\circ\Bigg[
\left(
\begin{matrix}
\gamma_{i_0}(x)+\gamma_{j_0}(y)\\
\gamma_{i_0}(x)-\gamma_{j_0}(y)
\end{matrix}\right)\cosh z
-\left(
\begin{matrix}
\gamma_{i_0}(x)-\gamma_{j_0}(y)\\
\gamma_{i_0}(x)+\gamma_{j_0}(y)
\end{matrix}\right)\sinh z
\Bigg]
\end{align*}
where
$$
A=\left[
\begin{matrix}
\frac{1}{2}(P+Q) & \frac{1}{2}(P-Q)\\
\frac{1}{2}(P-Q) & \frac{1}{2}(P+Q)
\end{matrix}\right].
$$
Hence $f$ is (up to $\Jt$-invariant affine transformation $A$) equivalent to
$$
f_0(x,y,z)=
\left(
\begin{matrix}
\gamma_{i_0}(x)+\gamma_{j_0}(y)\\
\gamma_{i_0}(x)-\gamma_{j_0}(y)
\end{matrix}\right)\cosh z
-\left(
\begin{matrix}
\gamma_{i_0}(x)-\gamma_{j_0}(y)\\
\gamma_{i_0}(x)+\gamma_{j_0}(y)
\end{matrix}\right)\sinh z.
$$
Now taking different combinations of $i_0,j_0$ we easily obtain (\ref{eq::3dim::1})--(\ref{eq::3dim::4}).
\end{proof}

\begin{uw} Note that (\ref{eq::3dim::3}) and (\ref{eq::3dim::4}) are affinely equivalent but the affine transformation
mapping (\ref{eq::3dim::3}) onto (\ref{eq::3dim::4}) is not $\Jt$-invariant.
\end{uw}

\begin{uw}\label{uw::Flat}
It is worth to mention that hypersurfaces from Theorem \ref{tw::klasyfikacjaSfer3dim} are flat and have parallel cubic form.
Actually they are the only proper $3$-dimensional affine hyperspheres with this property (see \cite{HL,DV2} for details).
\end{uw}

To conclude this section, we give an example of a $\Jt$-tangent affine hypersphere with a non-involutive distribution $\DD$.
\begin{ex}
Let $f$ be defined as follows:
$$
f\colon\R^3\ni(x,y,z)\mapsto\left(
\begin{matrix}
xy+1\\
x+\frac{1}{2}y\\
xy\\
x-\frac{1}{2}y
\end{matrix}
\right)\cosh z-
\left(
\begin{matrix}
xy\\
x-\frac{1}{2}y\\
xy+1\\
x+\frac{1}{2}y
\end{matrix}
\right)\sinh z
\in\R^4.
$$
It is not difficult to check that $f$ is an immersion and the vector field
$
C\colon\R^3\ni(x,y,z)\mapsto-f(x,y,z)\in\R^4
$ is transversal to $f_{\ast }(\R^3)$.
 \par In the canonical basis $\{\frac{\partial}{\partial x},\frac{\partial}{\partial y},\frac{\partial}{\partial z}\}$ the second fundamental form $h$ is expressed as follows
$$
h=\left[
\begin{matrix}
0 & -1 & 0\\
-1 & 0 & 2x\\
0 & 2x & -1
\end{matrix}
\right].
$$
The above implies that $f$ is nondegenerate.
By straightforward computations we obtain that $C$ is the affine normal field. Since $\Jt C=-f_z\in f_{\ast}(TM)$ it follows that $f$ is a $\Jt$-tangent affine hypersphere.
Moreover, we have that $\Jt f_x=f_x$, so $\frac{\partial}{\partial x}\in \DD ^{+}$. We also have $$\Jt (2x^2f_x+f_y+2xf_z)=-(2x^2f_x+f_y+2xf_z),$$ so the vector field $W:=2x^2\frac{\partial}{\partial x}+\frac{\partial}{\partial y}+2x\frac{\partial}{\partial z}$ belongs to $\DD^{-}$. Now, we compute that
$$h\Big(\frac{\partial}{\partial x},W\Big)=2x^2h\Big(\frac{\partial}{\partial x},\frac{\partial}{\partial x}\Big)+h\Big(\frac{\partial}{\partial x},\frac{\partial}{\partial y}\Big)+2xh\Big(\frac{\partial}{\partial x},\frac{\partial}{\partial z}\Big)=-1.$$
Using the formula (\ref{Wzory::eq::3}) and the above we get
\begin{align*}
\eta \Big(\Big[\frac{\partial}{\partial x},W\Big]\Big)=h\Big(\frac{\partial}{\partial x},\varphi W\Big)-h\Big(W,\varphi \frac{\partial}{\partial x}\Big)=-2h\Big(\frac{\partial}{\partial x},W\Big)=2.
\end{align*}
Since $\ker \eta=\DD$, the above implies that $[\frac{\partial}{\partial x}, W]\notin \DD$ and in consequence the distribution $\DD$ is not involutive.
\end{ex}

\section{Some applications}
In this section we show some applications of results obtained in the previous section. In particular, we show that $\widetilde{J}$-tangent affine hyperspheres
can be classified in terms, of so called, Calabi products.
\par Recall that (\cite{HLLV}) the Calabi product of two proper affine hyperspheres
$$
\psi_1\colon M_1\rightarrow\R^{n_1+1}\quad\text{and}\quad \psi_2\colon M_1\rightarrow\R^{n_2+1}
$$
is an affine immersion
$$
\psi\colon M_1\times M_2\times\R\rightarrow\R^{n_1+n_2+2}
$$
defined by the formula
$$
\psi(x,y,z):=(c_1e^{\sqrt{\frac{n_2+1}{n_1+1}}az}\psi_1(x),c_2e^{-\sqrt{\frac{n_2+1}{n_1+1}}az}\psi_2(y))
$$
where $c_1,c_2$ and $a$ are nonzero constants.
\par Let $f_1$ and $f_2$ and $f$ be the affine hyperspheres from Theorem \ref{tw::ParaComplexSpheresDecomposition}
with the affine normal fields $C_1=-\alpha f_1$, $C_2=-\beta f_2$ and $C=-\lambda f$, respectively ($\alpha, \beta, \lambda>0$).
Let us denote by $CP(f_1,f_2)$ the Calabi product of $f_1$ and $f_2$ with $c_1=2$, $c_2=1$ and $a=-1$. That is we have
$$
CP(f_1,f_2)(x,y,z):=(2e^{-z}f_1(x),e^{z}f_2(y))
$$
where $x=(x_1,\ldots,x_n)$ and $y=(y_1,\ldots,y_n)$. We shall always consider $CP(f_1,f_2)$ with a transversal vector field $C_{CP}:=-\lambda\cdot CP(f_1,f_2)$.
\par For the affine hypersphere $f_i$ ($i=1,2$) we shall denote by $\nabla_i, h_i, S_i, \Theta_i$ the Blaschke connection, the Blaschke metric, the shape operator
and the induced volume form, respectively. Similarly for the affine hypersurface $CP(f_1,f_2)$ we denote the induced affine objects by $\nabla^{CP}, h_{CP}, S_{CP}$
and $\Theta_{CP}$.
\par Using similar methods like in \cite{DV} one may obtain that the affine metric $h_{CP}$ of $CP(f_1,f_2)$ is given by the product metric
\begin{align}\label{eq::CalabiProductMetric}
h_{CP}=\frac{1}{2}\frac{\alpha}{\lambda}h_1\otimes\frac{1}{2}\frac{\beta}{\lambda}h_2\otimes(-\frac{1}{\lambda})dz^2
\end{align}
and the connection $\nabla^{CP}$ can be expressed in terms of $\nabla_1$, $\nabla_2$ and $h_1$, $h_2$ as follows:
\begin{align}\label{eq::CalabiConnecion}
{\nabla^{CP}}_{\partial _{x_i}}{\partial _{x_j}}&={\nabla_1}_{\partial _{x_i}}{\partial _{x_j}}+\frac{1}{2}\alpha h_1(\partial _{x_i},\partial _{x_j})\partial _{z}\\
{\nabla^{CP}}_{\partial _{y_i}}{\partial _{y_j}}&={\nabla_2}_{\partial _{y_i}}{\partial _{y_j}}-\frac{1}{2}\beta h_2(\partial _{y_i},\partial _{y_j})\partial _{z}\\
{\nabla^{CP}}_{\partial _{x_i}}{\partial _{y_j}}&={\nabla^{CP}}_{\partial _{y_j}}{\partial _{x_i}}=0\\
{\nabla^{CP}}_{\partial _{x_i}}{\partial _{z}}&={\nabla^{CP}}_{\partial _{z}}{\partial _{x_i}}=-\partial _{x_i}\\
{\nabla^{CP}}_{\partial _{y_i}}{\partial _{z}}&={\nabla^{CP}}_{\partial _{z}}{\partial _{y_i}}=\partial _{y_i}\\
{\nabla^{CP}}_{\partial _{z}}{\partial _{z}}&=0.
\end{align}
By straightforward computations we obtain
\begin{align}\label{eq::CalabiProductOmegaH}
\omega_{h_{CP}}=\sqrt{\frac{\alpha^n\beta^n}{2^{2n}\lambda^{2n+1}}}\omega_{h_1}\otimes\omega_{h_2}
\end{align}
and
\begin{align}\label{eq::CalabiVolumeForms}
\Theta_{CP}=(-2)^{n+2}\cdot\frac{\lambda}{\alpha\beta}\Theta_1\otimes\Theta_2.
\end{align}
\par Let us define a $(2n+2)\times (2n+2)$ matrix $A$ by the formula
$$
A:=
\left[\begin{matrix}
\frac{1}{2}I_{n+1} & I_{n+1}\\
\frac{1}{2}I_{n+1} & -I_{n+1}
\end{matrix}\right]
$$
where by $I_n$ we denote an  identity matrix of dimension $n\times n$. It is easy to see that $\det A=(-1)^{n+1}$ that is $A$ is an equiaffine transforamtion of $\R^{2n+2}$. By straightforward calculation one may check that
\begin{equation}\label{eq::FasCalabiProduct}
f = A\circ CP(f_1,f_2)
\end{equation}
that is we have the following
\begin{wn}\label{wn::CalabiProduct}
Let $f_1$, $f_2$ and $f$ be like in Theorem \ref{tw::ParaComplexSpheresDecomposition} then
$f$ is up to equiaffine transformation the Calabi product of $f_1$ and $f_2$.
\end{wn}
Since $A$ is an equiaffine transforamtion then both $f$ and $CP(f_1,f_2)$ are affine hyperspheres. In particular, using (\ref{eq::CalabiProductOmegaH}) and
(\ref{eq::CalabiVolumeForms}) we obtain
\begin{wn}\label{wn::RelationBetweenAlphaBetaLambda}
Let $f_1$, $f_2$ and $f$ be like in Theorem \ref{tw::ParaComplexSpheresDecomposition} and let $C_1=-\alpha f_1$, $C_2=-\beta f_2$ and $C=-\lambda f$
be their affine normals. Then $\alpha, \beta$ and $\lambda$ are related by the formula:
$$
\lambda=\Bigg[\frac{(\alpha\beta)^{n+2}}{2^{4n+4}}\Bigg]^\frac{1}{2n+3}.
$$
\end{wn}
\par Calabi products have many interesting properties. In particular, they preserve parallel cubic form (see \cite{DV}).
Moreover the affine metric of Calabi product is flat if and only if both components have a flat affine metric.
\par Now using (\ref{eq::FasCalabiProduct}) and results from \cite{MN,HL,HLV} one can obtain some classification results
for $\Jt$-tangent affine hyperspheres with the parallel cubic form. For example, when $\dim M = 5$, we have the following
\begin{wn}
Let $f\colon M\rightarrow \R^6$ be a $\Jt$-tangent affine hypersphere with an involutive distribution $\DD$ and a parallel cubic form. Then $f$ is locally affine equivalent to Calabi product
$CP(f_1,f_2)$ where $f_i$ ($i=1,2$) is one of the following surfaces:
\begin{align}
\label{eq::eli}x^2+y^2+z^2=1 \\
\label{eq::hip1}x^2+y^2-z^2=1 \\
\label{eq::hip2}x^2+y^2-z^2=-1 \\
\label{eq::4th}xyz=1 \\
\label{eq::5th}(x^2+y^2)z=1
\end{align}
\end{wn}
As it was already mentioned in previous section (see Remark \ref{uw::Flat}) all 3-dimensional $\Jt$-tangent affine hyperspheres with involutive distribution $\DD$
are flat. Of course this is not the case in higher dimensions. In particular, since the only 2-dimensional proper flat affine spheres are (\ref{eq::4th}) and (\ref{eq::5th}) (see \cite{MR}), taking Calabi products, we get the following classification result in 5-dimensional case:
\begin{wn}
If $f\colon M\rightarrow \R^6$ is a flat $\Jt$-tangent affine hypersphere with involutive distribution $\DD$ then $f$
is locally affine equivalent to one of the following hypersufaces:
\begin{align}
\label{eq::flat1}x_1x_2x_3x_4x_5x_6=1 \\
\label{eq::flat2}(x_1^2+x_2^2)(x_3^2+x_4^2)x_5x_6=1 \\
\label{eq::flat3}(x_1^2+x_2^2)x_3x_4x_5x_6=1
\end{align}
\end{wn}
Note, that contrary to $3$-dimensional case, not all $5$-dimensional flat proper affine hyperspheres are (after suitable affine transformation)
$\Jt$-tangent affine hyperspheres with the involutive distribution $\DD$. Indeed, it is well known that $(x_1^2+x_2^2)(x_3^2+x_4^2)(x_5^2+x_6^2)=1$
is a proper flat affine hypersphere but it is not affinely equivalent to any of (\ref{eq::flat1})--(\ref{eq::flat3}). Actually we will show
(see proof of Prop. \ref{prop::FlatJttangent}) that this hypersphere cannot be transformed into $\Jt$-tangent affine hypersphere.
\par Recall that we have the following general classification result (\cite{V})
\begin{tw}[\cite{V}]\label{tw::FlatHyperspheres}
Let $M$ be an affine hypersphere in $\R^{n+1}$ with constant sectional curvature $c$ and with nonzero Pick invariant $J$. Then $c=0$ and $M$
is equivalent to
\begin{equation}\label{eq::flatOdd}
(x_1^2\pm x_2^2)(x_3^2\pm x_4^2)\cdots(x_{2m-1}^2\pm x_{2m}^2)=1,
\end{equation}
if $n=2m-1$ or with
\begin{equation}\label{eq::flatEven}
(x_1^2\pm x_2^2)(x_3^2\pm x_4^2)\cdots(x_{2m-1}^2\pm x_{2m}^2)x_{2m+1}=1,
\end{equation}
if $n=2m$.
\end{tw}
Before we proceed with classification result for flat $\Jt$-tangent affine hyperspheres we shall show the following lemma
\begin{lm}\label{lm::Tranformation}
Let $f\colon M\rightarrow \R^{2n+2}$ be an affine hypersphere with the Blaschke field $C\colon M\rightarrow \R^{2n+2}$.
$f$ is affine equivalent to $\Jt$-tangent affine hypersphere if and only if there exists an affine transformation
$B\colon \R^{2n+2}\rightarrow \R^{2n+2}$ such that $B\circ C$ is tangent to $f$ and $B\sim\Jt$ (i.e. matrices for $B$ and $\Jt$ are similar)
\end{lm}
\begin{proof}
If $f$ is affine equivalent to $\Jt$-tangent affine hypersphere then there exists an affine transformation $A\colon \R^{2n+2}\rightarrow \R^{2n+2}$
such that $A\circ f$ considered with the transversal vector field $A\circ C$ is $\Jt$-tangent. That is if $X_1,\ldots,X_{2n+1}$ is a basis of vector fields on $M$
then
\begin{align*}
0&=\det [A\circ f_\ast X_1,\ldots,A\circ f_\ast X_{2n+1},\Jt\circ A\circ C]\\
 &=\det A\det [f_\ast X_1,\ldots,f_\ast X_{2n+1},A^{-1}\circ \Jt\circ A\circ C].
\end{align*}
Now it is enough to take $B:=A^{-1}\circ \Jt\circ A$. \par
On the other hand, when $B\sim\Jt$, there exists an invertible matrix $P$ such that
$B:=P^{-1}\circ \Jt\circ P$. Now taking $A:=P$ we prove the converse.
\end{proof}
Now we obtain
\begin{stw}\label{prop::FlatJttangent}
Let $f\colon M\rightarrow \R^{2n+2}$ be a flat $\Jt$-tangent affine hypersphere then distribution $\DD$ is involutive and
$f$ is affine equivalent to either
\begin{equation}\label{eq::flat2n+1::1}
(x_1^2\pm x_2^2)(x_3^2\pm x_4^2)\cdots(x_{2n+1}^2\pm x_{2n+2}^2)=1,
\end{equation}
if $n$ is odd or
\begin{equation}\label{eq::flat2n+1::3}
(x_1^2\pm x_2^2)(x_3^2\pm x_4^2)\cdots(x_{2n-1}^2\pm x_{2n}^2)x_{2n+1}x_{2n+2}=1,
\end{equation}
if $n$ is even.
\end{stw}
\begin{proof}
Since proper flat affine hyperspheres have nonzero Pick invariant, by Theorem \ref{tw::FlatHyperspheres}, they are affine
equivalent to (\ref{eq::flatOdd}) or (\ref{eq::flatEven}). In particular, $2n+1$ dimensional flat affine hyperspheres are equivalent to
\begin{equation}\label{eq::flat2n+1::2}
(x_1^2\pm x_2^2)(x_3^2\pm x_4^2)\cdots(x_{2n+1}^2\pm x_{2n+2}^2)=1.
\end{equation}
We shall show that all the above affine hyperspheres (with one exception) are, after suitable affine transformation, $\Jt$-tangent.
\par If $n$ is odd (\ref{eq::flat2n+1::2}) can be obtained as the Calabi product of two flat $n$-dimensional affine hyperspheres
$$
(x_1^2\pm x_2^2)\cdots(x_{n}^2\pm x_{n+1}^2)=1
$$
and
$$
(x_{n+2}^2\pm x_{n+3}^2)\cdots(x_{2n+1}^2\pm x_{2n+2}^2)=1
$$
and as such is affine equivalent to $\Jt$-tangent affine hypersphere with involutive distribution $\DD$.
\par If $n$ is even and at least one of "$\pm$" in (\ref{eq::flat2n+1::2}) is "$-$", without loss
of generality we may assume that (\ref{eq::flat2n+1::2}) contains  $(x_{2n+1}^2 - x_{2n+2}^2)$ term. Now applying affine transformation changing
$(x_{2n+1}^2 - x_{2n+2}^2)$ into $x_{2n+1}x_{2n+2}$ we can transform (\ref{eq::flat2n+1::2}) into (\ref{eq::flat2n+1::3}).
Since (\ref{eq::flat2n+1::3}) is the Calabi product of two flat $n$-dimensional affine hyperspheres of form (\ref{eq::flatEven})
it is affine equivalent to $\Jt$-tangent affine hypersphere with involutive distribution $\DD$.
\par Now it remained to show that for $n$ even
\begin{equation}\label{eq::NotJtTangent}
(x_1^2+ x_2^2)(x_3^2+ x_4^2)\cdots(x_{2n+1}^2+ x_{2n+2}^2)=1
\end{equation}
cannot be transformed by affine transformation into $\Jt$-tangent affine hypersphere. First note that (\ref{eq::NotJtTangent}) can be parameterized
as follows:
\begin{align*}
f(v_1,\ldots,v_{n+1},u_1,\ldots,u_n)=
\left(\begin{matrix}
e^{u_1}\cos v_1 \\
e^{u_1}\sin v_1 \\
\cdots \\
e^{u_n}\cos v_n \\
e^{u_n}\sin v_n \\
e^{-u_1-\cdots-u_n}\cos v_{n+1} \\
e^{-u_1-\cdots-u_n}\sin v_{n+1} \\
\end{matrix}\right)
\end{align*}
where $v_i, u_i\in\R$. Assume that $f$ is affine equivalent to $\Jt$-tangent affine hypersphere, then by Lemma \ref{lm::Tranformation}
there exists a matrix $B=[b_{ij}]\in GL(2n+2)$, $B\sim\Jt$ such that $B \circ f$ is tangent to $f$. That is
$$W:=\det[f_{v_1},\ldots,f_{v_{n+1}},f_{u_1},\ldots,f_{u_n},B\circ f]=0.$$
By straightforward (but quite long) computations one may obtain
\begin{align}\label{eq::W}
W=\big(-1\big)^{\frac{(n+1)(n+2)}{2}}\Big(\sum_{k=1}^n\sum_{s=1}^n e^{-u_k+u_s}A_{k,s}+e^{-(u_1+\cdots+u_n)}\sum_{k=1}^n e^{-u_k}A_{k,n+1}\\
\nonumber+e^{u_1+\cdots+u_n}\sum_{k=1}^n e^{u_k}A_{n+1,k}+A_{n+1,n+1}\Big),
\end{align}
where
\begin{align*}
A_{i,j}:=(\cos v_i\cos v_j b_{2i-1,2j-1}+\cos v_i\sin v_j b_{2i-1,2j}\\+\sin v_i\cos v_j b_{2i,2j-1}+\sin v_i\sin v_j b_{2i,2j})
\end{align*}
for $i,j = 1,\ldots,n+1$.
\par Since $W=0$ the above implies that $A_{k,s}=0$ for $k,s=1,\ldots,n$, $k\neq s$ and $A_{k,n+1}=A_{n+1,k}=0$ for $k=1,\ldots,n$.
In consequence we obtain
$$b_{2k-1,2s-1}=b_{2k-1,2s}=b_{2k,2s-1}=b_{2k,2s}=0$$
for $k,s=1,\ldots,n$, $k\neq s$  and
\begin{align*}
b_{2k-1,2n+1}&=b_{2k-1,2n+2}=b_{2k,2n+1}=b_{2k,2n+2}\\
&=b_{2n+1,2k-1}=b_{2n+2,2k-1}=b_{2n+1,2k}=b_{2n+2,2k}=0
\end{align*}
for $k=1,\ldots,n$. Moreover, from (\ref{eq::W}) we also have that $\sum_{k=1}^{n+1}A_{k,k}=0$ that is
\begin{align*}
&\sum_{k=1}^{n+1}\Big(\cos^2 v_{k}b_{2k-1,2k-1}+\sin v_{k}\cos v_{k}(b_{2k-1,2k}+b_{2k,2k-1})+\sin^2 v_{k}b_{2k,2k}\Big)\\
&=\sum_{k=1}^{n+1}\cos^2 v_{k}(b_{2k-1,2k-1}-b_{2k,2k})+\sum_{k=1}^{n+1}\sin v_{k}\cos v_{k}(b_{2k-1,2k}+b_{2k,2k-1})\\
&+\sum_{k=1}^{n+1}b_{2k,2k}=0.
\end{align*}
The above implies that $b_{2k-1,2k-1}=b_{2k,2k}$, $b_{2k-1,2k}=-b_{2k,2k-1}$ for $k=1,\ldots,n+1$ and $\sum_{k=1}^{n+1}b_{2k,2k}=0$.
Summarising, the matrix $B$ can be expressed as a block diagonal matrix
$$
B=
\left[\begin{matrix}
B_1 & 0 & \cdots & 0\\
0 & B_2 & \cdots & 0\\
\vdots & \vdots & \ddots & \vdots\\
0 & 0 & \cdots & B_{n+1}
\end{matrix}\right]
$$
where
$B_k=\left[\begin{matrix}
b_{2k,2k} & b_{2k-1,2k} \\
-b_{2k-1,2k} & b_{2k,2k}
\end{matrix}\right]$ for $k=1,\ldots,n+1$. Note that $\det B_k>0$ and in consequence $\det B=\det B_1\cdot\ldots\cdot \det B_{n+1}>0$.
On the other hand, since $B\sim\Jt$, we have $\det B=\det \Jt=(-1)^{n+1}=-1<0$, since $n$ is even, what contradicts our assumption.
\end{proof}
Let $J$ be the standard complex structure on $\R^{2n+2}\equiv \C^{n+1}$.
Although {\rm(\ref{eq::NotJtTangent})} cannot be transformed into $\Jt$-tangent affine hypersphere one may show that it is
affine equivalent to $J$-tangent affine hypersphere (more details on $J$-tangent affine hyperspheres can be found in \cite{SZ2}).
Actually we have the following general result
\begin{stw}
For every $n\geq 0$ the hypersurface
\begin{equation}\label{eq::JComplexTangent}
(x_1^2+ x_2^2)(x_3^2+ x_4^2)\cdots(x_{2n+1}^2+ x_{2n+2}^2)=1
\end{equation}
is (after suitable affine transformation) $J$-tangent affine hypersphere.
\end{stw}
\begin{proof}
Applying $P\colon\R^{2n+2}\ni(x_1,\ldots,x_{2n+2})\mapsto (x_1,x_{n+2},\ldots,x_{n+1},x_{2n+2})\in\R^{2n+2}$
to (\ref{eq::JComplexTangent}) we obtain
\begin{equation}\label{eq::JComplexTangent2}
(x_1^2+ x_{n+2}^2)(x_2^2+ x_{n+3}^2)\cdots(x_{n+1}^2+ x_{2n+2}^2)=1.
\end{equation}
Let us denote by $G$ the gradient of (\ref{eq::JComplexTangent2}). That is
$$
G:=\bigg[\frac{2x_1}{x_1^2+x_{n+2}^2},\ldots,\frac{2x_{n+1}}{x_{n+1}^2+ x_{2n+2}^2},\frac{2x_{n+2}}{x_{1}^2+ x_{n+2}^2},\ldots\frac{2x_{2n+2}}{x_{n+1}^2+ x_{2n+2}^2}\bigg]^T.
$$
Since $J(x_1,\ldots,x_{2n+2})=[-x_{n+2},\ldots,-x_{2n+2},x_1,\ldots,x_{n+1}]^T$ we see that $G$ is orthogonal to $J(x_1,\ldots,x_{2n+2})$
thus (\ref{eq::JComplexTangent2}) is a $J$-tangent affine hypersphere.
\end{proof}
\begin{uw}
The above results show that every proper $(2n+1)$-dimensional flat affine hypersphere is (after suitable affine transformation) either $\Jt$-tangent or $J$-tangent.
Moreover, when $n$ is odd (\ref{eq::JComplexTangent}) is both $\Jt$-tangent and $J$-tangent.
\end{uw}

\emph{This Research was financed by the Ministry of Science and Higher Education of the Republic of Poland.}
\bibliographystyle{aplain}

\vspace{1cm}
\par \ \\
Zuzanna Szancer\\
Department of Applied Mathematics, \\
University of Agriculture in Krakow, \\
253 Balicka St., 30-198 Krakow, Poland\\
e-mail: Zuzanna.Szancer@urk.edu.pl

\end{document}